\documentclass[12pt]{amsart}
\usepackage[top=30truemm,bottom=30truemm,left=30truemm,right=30truemm]{geometry}
\UseRawInputEncoding
\usepackage{mathrsfs}
\usepackage{bm}
\usepackage{amsfonts,amssymb}
\usepackage{dsfont}
\usepackage{extarrows}
\usepackage{amsmath}
\usepackage{mathrsfs}
\usepackage{enumerate}
\usepackage{amscd}
\usepackage[all]{xy}
\usepackage{hyperref}

\theoremstyle{plain}

\theoremstyle{definition} 

\newtheorem{thm}{Theorem}[section]
\newtheorem{cor}[thm]{Corollary}
\newtheorem{lem}[thm]{Lemma}
\newtheorem{prop}[thm]{Proposition}

\theoremstyle{definition}
\newtheorem{defn}{Definition}[section]

\theoremstyle{remark}
\newtheorem{rem}{Remark}[section]

\newcommand{\be}{\begin{equation}}
\newcommand{\ee}{\end{equation}}
\newcommand{\bea}{\begin{eqnarray}}
\newcommand{\eea}{\end{eqnarray}}
\newcommand{\ben}{\begin{eqnarray*}}
	\newcommand{\een}{\end{eqnarray*}}
\newcommand{\bt}{\begin{split}}
	\newcommand{\et}{\end{split}}
\newcommand{\bet}{\begin{equation}}
\newcommand{\mc}{\mathbb{C}}
\newcommand{\mr}{\mathbb{R}}
\newcommand{\ra}{\rightarrow}
\newcommand{\beq}{\begin{equation*}}
\newcommand{\eeq}{\end{equation*}}

\newcommand{\dbar}{\bar{\partial}}
\newcommand{\pa}{\partial}
\newcommand{\bp}{\bar{\partial}}

\begin{document}

\title[positivity Riemannian]
{Positivity of Riemannian metric
and the existence theorem of $L^2$ estimates for $d$ operator}

\author[X.Zhang]{Xujun Zhang}
\address{Xujun Zhang:\ 
Institute of Mathematics,
Academy of Mathematics and Systems Science, 
Chinese Academy of Sciences\\ 
Beijing 100190, P. R. China }
\email{xujunzhang@amss.ac.cn}

\begin{abstract}
        For the first main result of this note,
	we give a new characterization of
	Nakano $q$-positivity for the
	Riemannian flat vector bundle 
	in the sense of $L^2 $ estimate 
	for the $d$ operator.
	We also give a new characterization of
	the positivity of the curvature operator 
	associated to the Riemannian metric on   
	manifold in a local version in the sense of $L^2$ estimate.
	Then we prove two results 
	parallel to 
	Liu-Yang-Zhou's answer to  
	Lempert's question about Hermitian       
	holomorphic vector bundle.
	We also study the boundary convexity of 
	relative compact domain  
	in certain Riemann manifold 
	with convexity restriction. 
\end{abstract}

\maketitle

\tableofcontents

\section{Introduction}

Inspired by Deng-Ning-Wang-Zhou's work (\cite{DNW19},\cite{DNWZ}), 
Deng-Zhang
found a new characterization for convex function and curvature positivity for Riemannian flat bundle over the domain,
which applied to give a new proof of the classical Prekopa theorem and the matrix-valued Prekopa theorem.
They also give a new characterization of the convexity of smooth bounded domains in $\mr^n$  in the terms $L^2$ estimate condition for the $d$-equation
and a characterization of the pseudoconvexity of smooth bounded domains in $\mc^n$  in the terms $L^2$ estimate condition for the $\dbar$-equation (\cite{Deng-Zhang}).
The main purpose of this note is
to generalize the results in \cite{Deng-Zhang} to the partial positivity condition.

First, we give a new characterization 
for the partiality positivity 
for the Riemannian flat bundle.

\begin{thm}\label{thm:cha nakano positive}
Assume $D$ is a domain in $\mr^n$, 
$E=D \times \mr^r$ is the trivial vector bundle over $D$,
$g: D \ra  \text{Sym}^{+}_{r}(\mr)$ is the $C^2$ smooth Riemannian metric on $E$.
If for any strictly convex function $\psi$ on $\mr^n$,
for any $d$-closed $E$-valued $q$-form with compact support
$f \in \wedge^q_c(D, E) \cap \text{ker }(d)$,
there exists a $(q-1)$-form $u \in L^2_{q-1} (D, E)$ satisfy $du=f$ and
$$
\int_{D} \langle u,u \rangle  _{g} e^{-\psi} d\lambda   \leq  \int_{D}\langle \text{Hess}_{\psi}^{-1}  f,f \rangle  _{g} e^{-\psi}d\lambda ,
$$
provided the right hand side is finite,
then $(E, g)$ is Nakano $q$-semipositive.
\end{thm}

The notations in Theorem \ref{thm:cha nakano positive} will be introduced in \S \ref{sec:L2-estimate-bundle}.
The Nakano $q$-positivity of the Riemannian vector bundle is parallel to the Nakano $q$-positivity for the Hermitian holomorphic vector bundle which Siu introduced in \cite{SYT-1982}. 
Namely, when $E$ is line bundle,
the strictly Nakano $q$-positive of $E$ indicate that any $q$ eigenvalues of the Hessian of the weighted function on $E$ has positive sum.

When $q=1$, 
Theorem \ref{thm:cha nakano positive} can be viewed as applications of  Theorem 1.1 in \cite{DNW19} on the Riemannian flat line bundle.
Using the same methods in \cite{DNW19},  Inayama got similar results for Nakano $q$-positivity of Hermitian holomorphic vector bundle (\cite{Inayama}).

In \cite{Lempert}, Lempert asked whether a $C^2$ Hermitian metric whose curvature dominates $0$ is Nakano semi-positive in the usual sense.
Following Deng-Ning-Wang-Zhou's work in \cite{DNWZ},
Liu-Yang-Zhou answer Lempert's question affirmatively (\cite{Liu-Yang-Zhou}).
As an application of Theorem \ref{thm:cha nakano positive},
we prove the following Riemannian analogs of Liu-Yang-Zhou's results.

\begin{thm}\label{thm:real-version-lempert}
For any domain $D$ in $\mr^n$, 
$E=D\times \mr^r$ is trivial vector bundle over $D$, 
$\{g_{j} \}$ is a sequence of $C^2$ smooth Nakano $q$-semipositive Riemannian metric on $E$,
assume $\{g_{j}\}$ increasingly converge to a $C^2$ smooth Riemannian metric $g$, i.e., $g_{i} \leq g_{j} \leq g$ for any $i \leq j$,
then $(E, g)$ is Nakano $q$-semipositive definite.
\end{thm}

In the proof of Theorem \ref{thm:real-version-lempert}, 
an important ingredient is the following existence theorem of $L^2$ estimate of $d$ operator on Nakano $q$-positive flat Riemannian vector bundle.

\begin{thm}
	\label{thm:L^2 estimate for vector bundle}
	Assume $D$ is a convex domain in $\mr^n$,
	$E=D\times \mr^{r} \ra D$ is the flat vector bundle over $D$, 
	let $g: \mr^n \ra \text{Sym}^{+}_{r}(\mr)$ be a Riemannian metric on $E$.
	Assume $(E,g)$ is Nakano $q$-semipositive,
	then for any $d$-closed $q$-form $f\in \wedge^q (\bar{D}, E)\cap \text{Ker}(d) \cap \text{Dom}(d^*)$ with 
	$$
	\int_{\mr^n}\langle \Theta_{E}^{-1}f, f\rangle_g d\lambda<+\infty,
	$$ 
the equation $du=f$ can be solved with $u\in L^2_{q-1}(D, E)$ satisfying the following estimate:
	$$
	\int_{D}|u|^2_g d \lambda \leq \int_{D}\langle \Theta_{E}^{-1}f, f\rangle_g d\lambda.
	$$
\end{thm}

When $q=1$, $E$ is line bundle, Theorem \ref{thm:L^2 estimate for vector bundle} is due to the Brascamp-Lieb inequality (\cite{Bra-Lieb}).
We prove Theorem \ref{thm:L^2 estimate for vector bundle}
in \S \ref{sec:L2-estimate-bundle} based on H\"ormander's idea for $L^2$ estimate (\cite{Ber95},\cite{Hor65}).

Next, 
we study the positivity of the curvature operator of the Riemannian metric defined on the manifold.
For a Riemannian manifold $(M,g)$, any points $p \in M$,
the curvature operator $\mathscr{R}: \wedge^{2} T_{p}M \ra \wedge^{2} T_{p}M $ associated to the Riemannian metric $g$  is defined by:
\begin{equation}
\langle \mathscr{R}(X \wedge Y),Z\wedge W \rangle =\langle R(X,Y)W,Z \rangle,
\end{equation}
where $R$ is the Riemannian metric tensor.

It's well known that the positivity of the curvature 
operator $\mathscr{R}$ is equivalent to the positivity 
of the following differential operator: 
$$
\Re_q: \square-\triangle^{\wedge(T^{*}M)}_{0}: \wedge^{q} T^{*}_{p}M \ra \wedge^{q} T^{*}_{p}M , \forall 1 \leq q \leq n-1
$$
where $\square$ is the de-Rham Laplacian, $\triangle^{\wedge(T^{*}M)}_{0}$ is the Hodge Laplacian (\cite[Theorem 3.3]{Wu-Book-Bochner}).
We show that the local existence 
of $L^2$ existence of $d$ operator
induce 
the local semi-positivity 
of the curvature operator $\mathscr{R}$.

\begin{thm}\label{thm:main-results-manifold}
Let $(M,g)$ be a complete Riemannian manifold,
for any point $p \in M$,
if there exists a relative compact neighborhood $U$ of $x_0$,
for any smooth closed $1$-form $f \in \Omega^{1}_{c}(U)\cap \text{Ker}(d)$ with compact support, 
for any smooth function $\varphi$ defined on $\bar{U}$,
there exists a function $u \in L^2(U,\varphi)$ satisfy the equation $du=f$ 
with the estimate:
$$
\int_{U} | u|^2 e^{-\varphi} dV_{g} 
\leq 
\int_{U} \langle \text{Hess}^{-1}_{\varphi}f,f\rangle e^{-\varphi} dV_{g},
$$
provided the right hand side is finite,
then the curvature operator $\Re$ associated to the Riemannian metric $g$ is semipositive definitie at $p$.
\end{thm}

The notations in Theorem \ref{thm:main-results-manifold} is introduced 
in \S \ref{sec:Properties of curvature operator Re}.
In the proof of Theorem \ref{thm:main-results-manifold},
we choose $U$ be the normal coordinate neighborhood,
since there exists a canonical 
smooth strictly exhaustive convex function 
(square of the length function of geodesic) 
defined on $U$,
we can apply the localization technique to produce the contraction. 
Thus, if the manifold admits a
smooth strictly exhaustive convex function globally,
Theorem \ref{thm:main-results-manifold} can be generalized to the following condition without any technically change.

\begin{cor}\label{cor:postive-curvature-operator-manifold}
Let $(M,g)$ be a complete homogeneous $1$-convex Riemannian manifold,
if for any smooth closed $1$-form $f \in \Omega^{1}_{c}(M)\cap \text{Ker}(d)$ with compact support, 
for any smooth function $\varphi$ defined on $\bar{M}$,
there exists a function $u \in L^2(M,\varphi)$ satisfying $du=f$ and the following estiamtes
$$
\int_{M} |u|^2 e^{-\varphi} dV_{g} 
\leq 
\int_{M} \langle \text{Hess}^{-1}_{\varphi}f,f\rangle e^{-\varphi} dV_{g},
$$
provided the right hand side is finite,
then the curvature operator $\Re$ associated to the Riemannian metric $g$ is semipositive definitie on M.
\end{cor}

If a Riemannian manifold $M$ admits
a smooth strictly convex exhaustive function,
it's named as the $1$-convex Riemannian manifold 
(\cite{Harvey-Lawson-12},\cite{Harvey-Lawson-13}).
It was showed in \cite{Green-Wu-1976} that $1$-convex Riemannian
is diffeomorphic to Euclidean space.
The assumption of the homogeneous provided
that
there exists a 
smooth strictly convex 
exhaustive function 
such that for any point $p$ in $M$, 
$p$ is its nondegenerate critical point.
For the noncompact Riemannian manifold 
with positive sectional curvature,
Wu shows that the Busssman function of a geodesic ray is strictly
convex which is independent of the choice of the initial point (\cite{WUHH-acta}),
thus the restriction of homogeneous $1$-convex in 
Corollary \ref{cor:postive-curvature-operator-manifold}
and the rest of this note can be substituted
to noncompact Riemannian manifold with positive sectional curvature.

We can also prove an analogy result of Theorem \ref{thm:real-version-lempert} for the Riemannian metric on the Riemannian manifold.
 
\begin{thm}
\label{thm:deformation-Riemannian-metirc-manifold}
Assume $(M,g_t)$ is a complete homogeneous $1$-convex Riemannian manifold,
where $g_{t}$ is an increasing sequence of complete Riemannian metric with semi-positive definite the curvature operator $\mathscr{R}^{t}$ on $M$  
and converge to a Riemannian metric $g$,
then the curvature operator $\mathscr{R}$ associated to $g$ is semipositive definite on M.
\end{thm}

For the metric flow $\{g_t\}$ with properties
we defined in Theorem \ref{thm:deformation-Riemannian-metirc-manifold},
Hamilton shows the flow exists on any compact Riemannian manifold using 
nonlinear heat equation (\cite[Lemma 8.2]{Hamilton}).

Apply the same methods in \cite{Deng-Zhang},
we also show that the existence of the $L^2$ estimate indicates the boundary convexity of the domain in the Riemannian manifold.

\begin{thm}\label{thm:char-q-convex-manifold}
	Suppose $(M,g)$ is a $n$-dimensional complete homogeneous $1$-convex Riemannian manifold with semi-positive definite curvature operator $\Re$.
	Let $G$ be a relatively compact domain in $M$ with smooth boundary defining function $\tau: M \ra \mr$.
If for any closed $1$-form smooth up to the boundary $f \in \Omega^{q}(\bar{G})  \cap \text{Dom}(d^*) \cap \text{Ker}(d)$, 
for any strictly convex function $\psi$ on $M$,
there exists a solution $u\in L^2_{q-1}(G, \psi )$ satisfied  $du=f$
and
	$$
	\int_{G} |u|^2 e^{-\psi}d\lambda  \leq \int _{G} \langle  \text{Hess}^{-1}_{\psi} f, f \rangle e^{-\psi}d\lambda  ,
	$$	
provided the right hand side is finite,
then any $q$ eigenvalue of 
the restriction of 
the second fundamental form of 
$\pa G$ on the tangent space has nonnegative sum.
\end{thm}

If
any $q$ eigenvalue of the restriction of the second fundamental form of $\pa G$ on the tangent space has nonnegative sum,
domain $G$ is named as $q$-convex domain in the sense of Harvey-Lawson (\cite{Harvey-Lawson-12},\cite{Harvey-Lawson-13}), 
Sha, Wu showed that such a domain has the homotopy type of a CW-complex with dimension at most $q-1$ (\cite{Sha-1986},\cite{WuHH-1987}).
Theorem \ref{thm:char-q-convex-manifold} is nature generalization 
of Theorem 1.7 in \cite{Deng-Zhang} with a new construction 
of the weighted function.
We plan to consider the boundary pseudoconvexity
of the domain in Stein manifold
in subsequent work.

The rest of this note is organized as follows.
In \S \ref{sec:L2-estimate-bundle},
we prove 
an existence results 
of $L^2$ estimate.
Theorem \ref{thm:cha nakano positive}
and Theorem \ref{thm:real-version-lempert}
will be proved in \S \ref{sec: Partial Positivity of Riemannian flat vector bundle}.
In \S \ref{sec:Properties of curvature operator Re},
we recall some 
basic properties of 
the curvature operator of the Riemannian metric on the Riemannian manifold. 
We proved Theorem \ref{thm:main-results-manifold}
and Theorem \ref{thm:deformation-Riemannian-metirc-manifold}
in \S \ref{sec:main-results-manifold}.
In \S \ref{sec:boundary-Riemannian-manifold}, 
we discuss the boundary convexity of 
relative compact domain in $1$-convex  
Riemannian manifold 
to prove 
Theorem \ref{thm:char-q-convex-manifold}.

$\mathbf{Acknowlegements.}$
The methods to the main results 
in the present note are strongly inspired 
by the $L^2$ technique developed by Deng-Ning-Wang-Zhou
in \cite{DNW19},  \cite{DNWZ}.
The author is grateful to professor Fusheng Deng, his
Ph.D. advisor, for his valuable suggestion on related topics.
The author is now supported by the Academy of Mathematics and Systems Science.

\section{Differential Operators on Riemannian vector bundle}\label{sec:L2-estimate-bundle}

Some of the calculation technique in this section is inspired by \cite{Ji-Liu-Yu-2014}

For a domain $D$ in $\mr^n$, 
let $\pa D$ denote the boundary of $D$.
We assume $\pa D$ is smooth in the following sense.
In a neighborhood $U$ of $\pa D$, 
there exists a smooth real-valued function $\rho:U \ra \mr$ such that for any $x\in \pa D$, 
we have $\rho (x)=0$
and $|\nabla \rho(x) | \neq 0$.
We shall assume $\rho <0$ in $D$ and $\rho > 0$ out side $\bar D$.

$E=D \times \mr^r \ra  \mr^n$ is the trivial vector bundle over $\mr^n$,
we denote by $\wedge^{q}(D,E)=\wedge^{q}T^*_D \otimes E=\mathscr{C}^{\infty}(\wedge^{q}T^*D \otimes E)$
the space of $q$-form defined on $D$ with values in $E$.

$\wedge^{q}_{c}(D,E)$ denotes the elements in $\wedge^{q}(D,E)$
which has compact support.
$\wedge^{q}(\bar{D},E)$ denotes the elements in $\wedge^{q}(D,E)$ smooth up the boundary of $D$.

Assume $g:D\ra \text{Sym}^{+}_{r}(\mr)$ is a $C^2$ smooth Riemannian metric on $E$,
where $\text{Sym}^{+}_{r}(\mr)$ is the space of symmetric matrix with oder $r$.
Let $\{e_1,...,e_r \}$ denote the frame on $E$ satisfy $\langle e_{\tau},e_\xi\rangle=g_{\tau\xi}$,
for any $\alpha, \beta \in \wedge^{q}(D,E)$,
we have 
$$
\alpha=\sum^{r}_{\tau=1}(\sum_{|I|=q}\alpha_{I,\tau} dx_I)\otimes e_\tau, 
\beta=\sum^{r}_{\tau=1}(\sum_{|I|=q}\beta_{I,\tau} dx_I)\otimes e_\tau.
$$
we always assume the muti-index is increasing.

Let's recall the definition of the linear connection and its curvature  over $\wedge^{q}T^*_D \otimes E, \forall 0\leq  q \leq n-1$

\begin{defn}[\cite{Demaily-Book}]
The linear connection $\nabla$ on $(E,g)$ is a linear differential operator of order $1$ acting on $\wedge^{q}T^*_D \otimes E$
and satisfy the following properties 
\begin{enumerate}
\item $\nabla:  \wedge^{q}T^*_D \otimes E \ra \wedge^{q+1}T^*_D \otimes E, \forall 0\leq  q \leq n-1$

\item $\nabla(f \wedge s)= d f \wedge s +(-1)^{p} f \wedge \nabla s$.
for any $f \in  \wedge^{q}T^*_D \otimes E$
\end{enumerate}

The linear map induces the following bilinear map (covariant derivative):
$$
\begin{aligned}
\nabla  :\mathscr{C}^{\infty}(TD) \times \mathscr{C}^{\infty}(\wedge^{q}T^*D \otimes E) & \ra \mathscr{C}^{\infty} (\wedge^{q}T^*D \otimes E)  \\
     (X, \alpha) &\ra \nabla_{X} \alpha.
\end{aligned}
$$
The \emph{curvature of the connection $\nabla$} is the map $\mathscr{R}$ defined by:
\beq
 \begin{aligned}
\mathscr{R}: \wedge^2 \mathscr{C}^{\infty}(TD) \times \mathscr{C}^{\infty}&(\wedge^{q}T^*D \otimes E)  \ra \mathscr{C}^{\infty} (\wedge^{q}T^*D \otimes E)  \\
\mathscr{R}(X,Y) \alpha&=\nabla_{X}\nabla_{Y} \alpha-\nabla_{Y}\nabla_{X} \alpha-\nabla_{[X,Y]}\alpha
\end{aligned}
\eeq

\end{defn}

\begin{rem}
For any $X\in \mathscr{C}^{\infty}(TD)$, it's clear that the covariant derivative $\nabla_X$ in the direction of $X$ is a linear differential operator of order $0$.
\end{rem}
\begin{rem}
With the notations above,
the exterior differential operator can be expressed by
$$
d=\sum^{n}_{i=1} dx_{i} \wedge \frac{\pa}{\pa x_i}=\sum^{n}_{i=1} 
dx_{i} \wedge \nabla_{\frac{\pa}{\pa x_i}} 
$$
For any $\alpha \in  \wedge^{q}T^*D \otimes E$,  
a simple calculation shows that:
$$
d^2 \alpha= \mathscr{R} \alpha
$$
the flatness of the vector bundle $E$ 
(in our cases: the trivial bundle $D \times \mr^r$) 
provided $\mathscr{R}=0$, 
thus we have 
$$
d^2 \alpha= 0.
$$
\end{rem}

Next, we introduce the following point-wise inner products and global inner product over $\wedge^{q} (D,E), \forall  0\leq q \leq n$:
$$
\begin{aligned}
	\langle \alpha, \beta \rangle_{g}
	&=\sum_{\tau,\xi}\left (\sum_{I}\alpha_{I,\tau} \cdot \beta_{I,\xi}\cdot g_{\tau \xi}\right) \\
	\langle \langle \alpha,\beta \rangle \rangle_{g}&=\int_{D} \langle \alpha, \beta \rangle_{g} d\lambda
\end{aligned}
$$
and the corresponding norm:
$$
|\alpha|^2_{g}=\langle \alpha, \alpha \rangle_{g} \;\; \text{and} \;\; \|\alpha \|^2_{g}=\langle \langle \alpha,\alpha \rangle \rangle_{g}
$$
then we will denote by $L^2_{q}(D,E)$ the completion of $\wedge^{q}(D,E)$ with respect to the above global inner product.
We use the same notation $d$ to denote the maximal (weak) differential operator $d: L^2_{q-1}(D,E) \ra L^2_{q}(D,E)$ for any $1 \leq q \leq n$.

We consider the following chain of Hilbert spaces
$$
L^2_{q-1}(D,E) \xrightarrow{d} L^2_{q}(D,E) \xrightarrow{d} L^2_{q+1}(D,E) 
$$
and study the equation 
$$
du=f
$$
for any $f\in L^2_{q}(D,E)$ with $df=0$.

Assume 
$$
	u=\sum^{r}_{\tau=1}(\sum_{|K|=q-1}u_{K,\tau} dx_K)\otimes e_\tau , 
f=\sum^{r}_{\tau=1}(\sum_{|J|=q}f_{J,\tau} dx_J)\otimes e_\tau 
$$
where $ dx_K=\frac{\pa}{\pa x_{j}}\lrcorner dx_J$ for each $j$,
then the identity $du=f$, in the sense of distribution, means that
$$
\int_{D} u_{K,\tau}\cdot \frac{\pa \alpha_{J,\xi}}{\pa x_j} d\lambda=-\int_{D} f_{J,\tau}\alpha_{J,\xi} d\lambda
$$
for any $\alpha=\sum^{r}_{\tau=1}(\sum_{|J|=q}\alpha_{J,\tau} dx_I)\otimes e_\tau \in \wedge^{q}_{c}(D,E)$.
Substitute $\alpha_{J,\tau}$ with $g_{\tau\xi} \cdot \alpha_{J,\tau}$,
we have 
$$
\langle \langle u,  \sum_{i,j,\tau,\xi, J}( g_{\tau\xi}^{-1}
\frac{\pa g_{\tau\xi}}{\pa x_j}\alpha_{J,\tau}+\frac{\pa \alpha_{J,\tau}}{\pa x_j})\rangle\rangle_{g}=\langle \langle du,\alpha  \rangle\rangle_{g}
$$
this means that the formal adjoint of the operator $d:L^2_{q-1}(D,E) \ra L^2_{q}(D,E)$ is given by 
$$
d^*_{g}\alpha=-\sum^{r}_{\tau=1}\left( \sum_{i,j,\tau,\xi,J}  (g_{\tau\xi}^{-1}
\frac{\pa g_{\tau\xi}}{\pa x_j}\alpha_{J,j}+\frac{\pa \alpha_{J,\tau}}{\pa x_j})dx_{K} \right) \otimes e_{\tau}.
$$

\begin{defn}
For any $C^2$ smooth Riemannian metric $g:D\ra \text{Sym}^{+}_{r}(\mr)$ on $E$,
the \emph{curvature operator} $\Theta_{E}$ associated to the metric $g$
is defined by:
$$
\Theta_{E}=-\sum^{n}_{i,j=1}
\left( \sum_{\tau,\xi} \frac{\pa}{\pa x_{i}}(g^{-1}_{\tau \xi}\frac{\pa g_{\tau \xi}}{\pa x_{j}}) \right) dx_{i}\wedge\frac{\pa }{\pa x_j} \lrcorner
:\wedge^{q}(D,E) \ra \wedge^{q}(D,E), \forall 1 \leq q \leq n 
$$
\end{defn}

\begin{lem}
\begin{equation}
dd^{*}_{g}+d^{*}_{g}d=\Theta_{E}+T_{g}
\end{equation}
where 
$$
T_{g}=-\sum^{n}_{j=1}  \left( 
g^{-1}\frac{\pa g}{\pa x_j} 
\frac{\pa }{\pa x_j}-\frac{\pa}{\pa x_{j}} \frac{\pa}{\pa x_{j}}
\right) : 
$$
\end{lem}
\begin{proof}
$$
\begin{aligned}
dd^*_{g} &= -\sum_{i,j} dx_{i} \wedge \frac{\pa}{\pa x_{i}}(g^{-1}\frac{\pa g}{\pa x_j}\frac{\pa }{\pa x_j} \lrcorner +\frac{\pa }{\pa x_j} \lrcorner \frac{\pa }{\pa x_j}) \\
 &=-\sum \left( 
 dx_{i}\wedge \frac{\pa }{\pa x_{i}}(g^{-1}\frac{\pa g}{\pa x_j})\frac{\pa }{\pa x_j}\lrcorner 
 +g^{-1}\frac{\pa g}{\pa x_j} dx_i\wedge \frac{\pa }{\pa x_j} \lrcorner \frac{\pa }{\pa x_i}+dx_i \wedge \frac{\pa }{\pa x_j} \lrcorner \frac{\pa^2 }{\pa x_i \pa x_j}
 \right)
\end{aligned}
$$

$$
\begin{aligned}
d^*_{g} d &=-\sum 
(
g^{-1} \frac{\pa g}{\pa x_j} \frac{\pa }{\pa x_j}\lrcorner 
+\frac{\pa }{\pa x_j} \lrcorner \frac{\pa }{\pa x_j}
)
(dx_i \wedge \frac{\pa }{\pa x_i}) 
\\
 &= -\sum\left(g^{-1}\frac{\pa g}{\pa x_j} 
\frac{\pa }{\pa x_j}
-g^{-1}\frac{\pa g}{\pa x_j} dx_i\wedge \frac{\pa }{\pa x_j} \lrcorner \frac{\pa }{\pa x_i} -\frac{\pa^2 }{(\pa x_{j})^2} -dx_i \wedge \frac{\pa }{\pa x_j} \lrcorner \frac{\pa^2 }{\pa x_i \pa x_j}
 \right) 
\end{aligned}
$$
Then we have:
$$
\begin{aligned}
d^*_{g} d +dd^*_{g} &= \Theta_{E} 
-\sum^{n}_{j=1}  \left( 
g^{-1}\frac{\pa g}{\pa x_j} 
\frac{\pa }{\pa x_j}-\frac{\pa}{\pa x_{j}} \frac{\pa}{\pa x_{j}}
\right) \\
&=\Theta_{E}+T_{g}
\end{aligned}
$$
\end{proof}

\begin{rem}
For a smooth real-valued function $\varphi:\mr^n \ra \mr $ defined on $\mr^n$,
we denotes $\text{Hess}(\varphi)$ the Hessian matrix:
$$
\text{Hess}(\varphi)
=\left( \frac{\pa^2 \varphi}{\pa x_i \pa x_j} 
\right)_{1\leq i,j\leq n}.
$$

Similar to $\Theta_E$, 
we introduce 
the following linear operator 
defined on $\wedge^{q}T^{*}_{\mr^n}$:
\begin{equation}
	\text{Hess}_{\varphi} =\sum_{1\leq i,j \leq n} \varphi_{ij} dx_i \wedge \frac{\pa }{\pa x_{j}} \lrcorner.
\end{equation}
\end{rem}

We recall the definition of smooth $q$-convex function in the sense of Harvey and Lawson (we write it as $q$-convex function for simplicity).

\begin{defn}
	For any smooth function  $\varphi:\mr^n \ra \mr$ defined on $\mr^n$,
	we says that $\varphi$ is strictly  $q$-convex 
	(respectively $q$-convex) 
	if the sum of any $q$ eigenvalues of 
	$\text{Hess}(\varphi)$ is
	 positive (respectively semi-positive).
\end{defn}

\begin{rem}
It's clear that if $\varphi$ is $q$-convex,
the operator $\text{Hess}_{\varphi}$ 
is semipositive definite on $\wedge^{q}$.
\end{rem}

\begin{defn}[\cite{Harvey-Lawson-12},\cite{Harvey-Lawson-13}]
	A domain $D$ with smooth boundary is called $q$-convex domain if the boundary-defining function is $q$-convex.
\end{defn}

Similar to Lemma 7.2 in \cite{Ji-Liu-Yu-2014},
we have the following identity.

\begin{lem}
For any $u\in \wedge^{q-1}(\bar{D}, E),\alpha \in \wedge^{q}(\bar{D},E) \cap \text{Dom}(d^*_g)$,
we have:
\begin{equation}\label{eq:stokes-formula}
\langle \langle  du, \alpha  \rangle \rangle _g=
\langle \langle u, d^{*}_{g} \alpha  \rangle \rangle _g +\int_{\pa D} \langle u, \nabla \rho 
\lrcorner \alpha \rangle_{g} dS 
\end{equation}
\end{lem}

The equation \eqref{eq:stokes-formula} shows that:
$$
\alpha \in \wedge^{q}(\bar{D},E) \cap \text{Dom}(d^*_g) \Longleftrightarrow \nabla \rho 
\lrcorner \alpha=0 \;\; on \;\; \pa D
$$

Now we are ready to prove the following identity.

\begin{prop}\label{prop:Bochenr-identity-boundary-term-bundle}
	Assume $D$ is a bounded domain in $\mr^n$ with smooth boundary,
	$E=D\times \mr^r \ra D$ is the flat vector bundle over $D$,
	$g:D \ra \text{Sym}^{+}_{r}(\mr)$ is a $C^2$ smooth Riemannian metric on $E$.
	Then for any $\alpha \in \wedge^{q}(D,E) \cap \text{Dom}(d^*)$, we have
	 \begin{equation}\label{eq:Bochner-identity-boundary-vector-bundle}
	 	\begin{aligned}
	\| d\alpha \| ^2_{g}
	+\| d^{*}_{g}\alpha \| ^2_{g}&=\| \nabla \alpha \|^2_{g} 
	+\int_{D}\langle  \Theta_{E}\alpha ,\alpha \rangle  _{g} d\lambda  +\int_{\pa D} \langle \text{Hess}_{\rho} \alpha,  \alpha \rangle_g dS	
	 	\end{aligned}
	 \end{equation}
\end{prop}

\begin{proof}
We only need to verify the following identity:
$$
\langle \langle T_{g}\alpha, \alpha  \rangle \rangle_g =	\| \nabla \alpha \| ^2_{g}+\int_{\pa D} \langle \text{Hess}_{\rho}\alpha,  \alpha \rangle_g dS	
$$
Since 
$$
\begin{aligned}
T_g \alpha &=
-\sum^{n}_{j=1}  \left( 
g^{-1}\frac{\pa g}{\pa x_j} 
\frac{\pa \alpha }{\pa x_j}-\frac{\pa}{\pa x_{j}} \frac{\pa}{\pa x_{j}}\alpha \right)  \\
&=\sum^{n}_{j=1}  \delta_j \frac{\pa }{\pa x_j} \alpha 
\end{aligned}
$$
where:
$$
\delta_j =
-\sum^{n}_{j=1}  \left( 
g^{-1}\frac{\pa g}{\pa x_j} 
-\frac{\pa}{\pa x_{j}} \right) .
$$
Set
$$
X=\langle \alpha, \frac{\pa \alpha}{\pa x_j}  \rangle_g \frac{\pa }{\pa x_j}
$$
then we have:
$$
\begin{aligned}
\text{div} X &=\frac{\pa }{\pa x_j} \langle \alpha, \frac{\pa \alpha}{\pa x_j}  \rangle_g  \\
&=\sum_{I,\xi,\tau }  \left( 
\alpha_{I,\tau}\cdot g_{\tau \xi} \cdot (g^{-1}_{\tau \xi}\frac{\pa g_{\tau \xi}}{\pa x_j} 
\frac{\pa \alpha_{I,\xi} }{\pa x_j}-\frac{\pa \alpha_{I,\xi}}{\pa x_{j}} \frac{\pa}{\pa x_{j}} )+\frac{\pa \alpha_{I,\tau}}{\pa x_j} \cdot g_{\tau \xi} \cdot \frac{\pa \alpha_{I,\xi}}{\pa x_j}
\right)
  \\ 
&=-\langle \delta_j \frac{\pa }{\pa x_j} \alpha, \alpha    \rangle_{g} +\langle\frac{\pa }{\pa x_j} \alpha,\frac{\pa }{\pa x_j} \alpha \rangle _{g}
\end{aligned}
$$
then we have 
$$
\int_{D} \text{div} X =\int_{\pa D} \langle \nabla_{\nabla \rho} \alpha ,\alpha \rangle_{g} dS 
$$

Since $\alpha \in \wedge^{q}(\bar{D},E) \cap \text{Dom}(d^*_g) $
we have
$ \nabla \rho 
\lrcorner \alpha=0 \;\; on \;\; \pa D $,
thus 
$$
d( \nabla \rho 
\lrcorner \alpha)= (d \nabla \rho 
\lrcorner) \alpha+ \nabla \rho 
\lrcorner d \alpha=0 \;\;  on \;\; \pa D
$$
From
\begin{equation}\label{eq:Lie-derivative-vector-bundle}
(d \nabla \rho 
\lrcorner) \alpha+ \nabla \rho 
\lrcorner d \alpha- \nabla_{\nabla \rho}=\text{Hess}_{\rho}, 
\end{equation}
we have 
$$
\langle \langle T_{g}\alpha, \alpha  \rangle \rangle_g =	\| \nabla \alpha \| ^2_{g}+\int_{\pa D} \langle \text{Hess}_{\rho}\alpha,  \alpha \rangle_g dS	
$$
\end{proof}

\begin{rem}
The identity \eqref{eq:Lie-derivative-vector-bundle} is proved in \cite[Lemma 7.1]{Ji-Liu-Yu-2014}.
\end{rem}

\begin{defn}\label{defn:Nakano-q-positivity-vector-bundle}
Let $(E,g)$ be a Riemannian flat vector bundle over domain $D$ with $C^2$ smooth Riemannian metric $g$,
$\Theta_{E}$ is the curvature operator associated to the metric
$g$,
for any integral $1 \leq q \leq n$,
we say that $(E,g)$ is \emph{Nakano $q$-positive} (respectively \emph{semipositive}) if the inner product:
\be
\langle \Theta_{E} \alpha, \alpha  \rangle 
\ee
is positive (respectively \emph{non-negative}) definite at any point $x \in D$ fot any no-zero elements $\alpha \in \wedge^{q}(D,E)$.
\end{defn}

Now we are ready to prove the following $L^2$ estimate theorem for the $d$ operator.

\begin{thm}
Assume $D$ is a convex domain in $\mr^n$,
$E=D\times \mr^{r} \ra D$ is the flat vector bundle over $D$, 
let $g: \mr^n \ra \text{Sym}^{+}_{r}(\mr)$ be a Riemannian metric on $E$.
Assume $(E,g)$ is Nakano $q$-semipositive,
then for any $d$-closed $q$-form $f\in \wedge^q (\bar{D}, E)\cap \text{Ker}(d) \cap \text{Dom}(d^*)$ with 
$$
\int_{\mr^n}\langle \Theta_{E}^{-1}f, f\rangle_g d\lambda<+\infty,
$$ 
the equation $du=f$ can be solved with $u\in L^2_{q-1}(D, E)$ satisfying the following estimate:
$$
\int_{D}|u|^2_g d \lambda \leq \int_{D}\langle \Theta_{E}^{-1}f, f\rangle_g d\lambda.
$$
\end{thm}

\begin{proof}
	
Assume $D$ is a bounded domain with smooth boundary,
since $D$ is convex, 
by the Bochner-type identity \eqref{eq:Bochner-identity-boundary-vector-bundle},
we have 
\begin{equation}\label{eq:Bochner-identity-boundary-vector-bundle-L2}
\langle \langle \Theta_{E} \alpha, \alpha \rangle \rangle 
\leq \|  d^*_g \alpha\|^2_{g}+\| d\alpha \|^2_{g}
\end{equation}
holds for $\forall \alpha \in \wedge^{q}(D, E) \cap \text{Dom}(d^*).$
The completeness of the space $L^2_{q}(D, E)$ indicates the above inequality can be extended to arbitrary forms $\alpha$ such that $ \alpha , d \alpha , d^*_g \alpha \in L^2$.
Now we consider the following Hilbert space orthogonal decomposition:
$$
L^2_{q}(D,E)=\text{Ker}(d) \oplus \left(\text{Ker}(d) \right)^{\bot},
$$
notice that the space $\text{Ker}(d)$ is closed.

For any $v=v_1 +v_2 \in L^2_q$ where $v_1 \in \text{Ker}(d) $ and $v_2 \in  \left(\text{Ker}(d) \right)^{\bot}$
,
since $\left(\text{Ker}(d) \right)^{\bot} \subset \text{Ker}(d^*)$,
we have $d^* \alpha =0$,
and 
$$ 
\langle\langle f,v \rangle \rangle ^2_{g} =\langle\langle f,v_{1} \rangle \rangle^2_{g} \leq \left(\int_{D} \langle\Theta_{E}^{-1}f,f \rangle d\lambda \right)  \cdot \left(\int_{D} \langle\Theta_{E}v_1,v_1 \rangle d\lambda \right) ,
$$
from the inequality \eqref{eq:Bochner-identity-boundary-vector-bundle-L2}, we have 
$$
\int_{D} \langle\Theta_{E}v_1,v_1 \rangle d\lambda \leq \| d^*_g v_1\|^2_{g}=\|d^*_g v \|^2_{g}.
$$
Combining both inequalities,
we have 
$$
\langle\langle f,v \rangle \rangle ^2_{g}  \leq \left(\int_{D} \langle\Theta_{E}^{-1}f,f \rangle d\lambda \right) \cdot \|d^*_g v \|^2_{g}
$$
holds for any $v \in L^2$.
This shows that we have a well defined 
linear form:
$$
d^*v \ra \langle  \langle v ,f\rangle \rangle
$$
here $v \in  \text{Dom}(d^*) \cap L^2_{q}$. By the Hahn-Banach theorem, there exists an element $u \in L^2_{q-1}(D,E)$ with 
$$
\|u\|^2_g \leq \int_{D} \langle\Theta_{E}^{-1}f,f \rangle d\lambda
$$
such that $\langle\langle f,v \rangle \rangle _{g}=\langle\langle u,d^*_{g}v \rangle \rangle ^2_{g}$

For the general cases,
since $D$ admits a smooth strictly $q$-convex exhaustive function,
we have 
$$
D_{c}=\{x\in D; \psi(x) \leq c\} \subset \subset D  
$$
for all $c \in \mr$.

For each $c$, 
the equation $du=f$ has solution $u \in L^2_{q-1}(D,E)$  with 
$$
\int_{D_c}|u|^2_g d \lambda \leq \int_{D_c}\langle \Theta_{E}^{-1}f, f\rangle_g d\lambda.
$$
It's clear that the integral on the right hand side is monotonic increasing and bounded above:
$$
\int_{D_{c}} \langle \Theta_{E}^{-1} f,f \rangle_{g} d\lambda \leq \int_{D} \langle \Theta_{E}^{-1} f,f \rangle_{g} d\lambda, \forall c \in \mr.
$$
Thus $\{u_c\}$ is uniformly  bounded above in $L^2_{q}(D_c,E)$ for any $c \in \mr$,
by the Algaoglu's Theorem, 
there exists a subsequence $\{u_{c_k} \}$ of  $\{u_c\}$ converging weakly to $u$ in $L^2$.
We have $du=f$ and 
$$
\int_{D} |u|^{2}_{h} d\lambda \leq \underset{k \ra +\infty}{\lim \inf} \int_{D_{c_k}}  |u_{c_k}|^{2}_{h} d\lambda  \leq \int_{D} \langle \Theta_{E}^{-1} f,f \rangle_{g} d\lambda
$$
for $c_k \ra +\infty$.
\end{proof}

\section{Proof of Theorem \ref{thm:cha nakano positive} and Theorem \ref{thm:real-version-lempert}}
\label{sec: Partial Positivity of Riemannian flat vector bundle}

In this section we prove Theorem \ref{thm:cha nakano positive} and Theorem \ref{thm:real-version-lempert}, namely, we give a new characterization of the Nakano $q$-semipositivity of Riemannian vector bundle and 
we also show that any Riemannian metrics dominate to zero is Nakano $q$-semipositivity in the usual sense.

\begin{thm}(=Theorem \ref{thm:cha nakano positive})
Assume $D$ is a domain in $\mr^n$, 
$E=D \times \mr^r$ is the trivial vector bundle over $D$,
$g: D \ra  \text{Sym}^{+}_{r}(\mr)$ is the $C^2$ smooth Riemannian metric on $E$.
If for any strictly convex function $\psi$ on $\mr^n$ of the form
$$\psi(x)=a\|x-x_0\|^2-b,\ (\forall a,b>0, \forall x_0 \in \mr^n),$$
for any closed $E$-valued $q$-form with compact support
$f \in \wedge^q_c(D, E) \cap \text{Ker }(d)$,
there exists $u \in L^2_{q-1} (D, E)$ satisfy $du=f$ and
$$
\int_{D} \langle u,u \rangle  _{g} e^{-\psi} d\lambda   \leq  \int_{D}\langle \text{Hess}_{\psi}^{-1}  f,f \rangle  _{g} e^{-\psi}d\lambda ,
$$
provided the right hand side is finite,
then $(E, g)$ is Nakano $q$-semipositive.
\end{thm}

\begin{proof}
For any $\alpha\in \wedge^q_{c}(D, E)$, 
we have
$$
\begin{aligned}
 \langle  \langle \alpha ,f \rangle   \rangle_g^2 &= \langle  \langle \alpha ,du \rangle   \rangle_g^2=\langle   \langle   d^*_{g}\alpha ,u \rangle  \rangle_g^2 \\
&\leq \langle   \langle   d^*_{g}\alpha ,d^*_{g}\alpha \rangle  \rangle_g  \cdot  \langle   \langle   u ,u \rangle  \rangle_g  .
\end{aligned}
$$
Combing this with Proposition \ref{prop:Bochenr-identity-boundary-term-bundle}, 
we have
$$
\begin{aligned}
\langle   \langle  \alpha ,f \rangle  \rangle_g^2 
& \leq   \int_{D}\langle  \text{Hess}_{\psi}^{-1}f,f \rangle _{g} e^{-\psi} d\lambda   \\
& \times \left(\int_{D} | \nabla \alpha |^2 _{g}  e^{-\psi} d\lambda +\int_{D}  \langle  (\Theta_{E}+\text{Hess}_{\psi})\alpha,\alpha \rangle _{g} e^{-\psi} d\lambda -\int_{D}|d\alpha|^2e^{-\psi} d\lambda\right)\\
& \leq \int_{D}\langle  \text{Hess}_{\psi}^{-1}f,f \rangle _{g} e^{-\psi} d\lambda \\  
&\times\left(\int_{D} | \nabla \alpha |^2 _{g}  e^{-\psi} d\lambda +\int_{D}  \langle  (\Theta_{E}+\text{Hess}_{\psi})\alpha,\alpha \rangle _{g} e^{-\psi} d\lambda \right),
\end{aligned}
$$
Setting $\alpha=\text{Hess}_{\psi}^{-1}f,$ then we obtain the following inequality:
\begin{equation}\label{main-inequality-matrix}
\int_{D} | \nabla \alpha |^2 _{g}  e^{-\psi} d\lambda +\int_{D}  \langle  \Theta_{E} \alpha,\alpha \rangle _{g} e^{-\psi} d\lambda \geq 0
\end{equation}

Now we argue by contradiction. 
Suppose $(E, g)$ is not Nakano $q$-semipositive, 
then there exist $p_{0} \in D$, constants $a,c>0$,
 and an $E$-valued $q$-form with constant coefficients:
 $$ 
 \xi=\sum^{r}_{\tau=1}(\sum_{|J|=q}\xi_{J,\tau}dx_J) \otimes e_r \in \wedge^{q}(D,E)
 $$ 
 such that 
 $$
 \langle \Theta_{E} \xi, \xi  \rangle_g(x) <-c<0
 $$
 holds for all $x \in B(p_0, a) \subset D$,
 we assume $p_0=O=(0,..,0)$ for simplicity and write $B(O, a)$ as $B_a$.
 
 Take $\chi(x) \in C^{\infty}_{c}(B_a)$ such that $\chi(x)=1$ for all $x\in B_{\frac{a}{2}}$,
 set
 $$
 \nu(x)=\chi(x)\cdot \sum^{r}_{\tau=1}(\sum_{|K|=q-1}x_j\xi_{J,\tau}dx_K) \otimes e_r
 $$
 where $dx_j\wedge dx_K=dx_J$.
Take $f=dV_g \in \wedge^{q}_c(D,E)$,
we have $df=0$ and 
$$
f(x)=\sum^{r}_{\tau=1}(\sum_{|J|=q}\xi_{J,\tau}dx_J) \otimes e_r 
$$
for all $x\in B_{\frac{a}{2}}$.

For any $s>0$, we set $\psi_{s}(x)=s(|x|^2-\frac{a^2}{4})$,
then $\psi_s$ is a sequence of
smooth strictly convex function on $\mr^n$,
set 
$$
\alpha_s =\text{Hess}_{\psi_s}^{-1}f=\frac{q}{2s}f\in \wedge^{q}_{c}(D,E).
$$
Since $f$ has compact support, there exists a constant $C>0$ such that
$|\nabla \alpha_{s} (x) |^2 \leq \frac{C}{s}$ 
for all $x\in D, \; \forall s>0$ and $|\nabla \alpha_s(x) |=0$ for all $x\in B_{\frac{a}{2}}$ since the coefficients of $\xi$ is constant. 

Substitute $\alpha,\psi$ with  $\alpha_s, \psi_s$ in \eqref{main-inequality-matrix},
we have
$$
\begin{aligned}
& s^2\left( \int_{D} \langle \Theta_{E} \alpha_{s},\alpha_{s} \rangle _{g}e^{-\psi_{s}}d\lambda 
+\int_D |\nabla \alpha_s |^2 _{g} e^{-\psi_{s}}d\lambda\right)  \\
\leq 
&   -c \frac{q^2}{4} \int_{B_{\frac{a}{2}}} e^{-\psi_{s}} d\lambda
+s^2\int_{D\backslash B_{\frac{a}{2}}} \langle \Theta_E \alpha_s,\alpha_s \rangle _{g} e^{-\psi_{s}}d\lambda 
\\
&+C^2\int_{D\backslash B_{\frac{a}{2}} } e^{-\psi_{s}} d\lambda+\int_{B_{\frac{a}{2}}} 0d\lambda \\ 
\end{aligned}
$$
since $\underset{s \ra + \infty}{lim} \psi_{s} (x) =+\infty, \forall x \in D \backslash B_{\frac{a}{2}}$ and $\psi_s(x) \leq 0$ for all $x \in B_{\frac{a}{2}} $ and $\forall s >0$,
thus 
$$
\int_{D} | \nabla \alpha_s |^2 _{g}  e^{-\psi} d\lambda +\int_{D}  \langle  \Theta_{E} \alpha_s,\alpha_s \rangle _{g} e^{-\psi} d\lambda  <0
$$
for $s \ra +\infty$, 
this contradicts to \eqref{main-inequality-matrix},
thus $(E,g)$ is Nakano $q$-semipositive.
 \end{proof}

Next we prove Theorem \ref{thm:real-version-lempert},
first, we need to clarify the definition of the increasing metric.

\begin{defn}
	For any domain $D$ in $\mr^n$, 
$E=D\times \mr^r$ is trivial vector bundle over $D$, $\{g_{j} \}$ is a sequence of $C^2$ smooth
Riemannian metric,
we say that $\{g_{j}\}$ is increasing
if for any $f \in \wedge^{q}(D,E)$
we have 
$$
\| f \|^2_{g_i} \leq \| f \|^2_{g_j}
$$
for any $i \leq j$.
\end{defn}

\begin{thm}(=Theorem \ref{thm:real-version-lempert})
For any domain $D$ in $\mr^n$, 
$E=D\times \mr^r$ is trivial vector bundle over $D$, $\{g_{j} \}$ is a sequence of $C^2$ smooth Nakano $q$-semipositive Riemannian metric,
assume $\{g_{j}\}$ increasingly converge to a $C^2$ smooth Riemannian metric $g$, 
then $(E, g)$ is Nakano $q$-semipositive.
\end{thm}

\begin{proof}
	By the assumption,
	we have 
	$$
	\langle \langle \text{Hess}_{\psi}^{-1}f,f \rangle\rangle_{g_{j}} \leq	\langle \langle \text{Hess}_{\psi}^{-1}f,f \rangle\rangle_{g}
	$$
holds for all $f\in L^2_{q}(D,E)\cap \text{Ker}(d)$ and any $j \in \mathbb{N}^*$, provided the right hand side is finite.

For each $ j$,
Theorem \ref{thm:L^2 estimate for vector bundle} shows that,
there exists $u_j \in L^2_{q}(D,E)$ such that $du_j=f$ with 
$$
 \int_{D} | u_{j}|^2_{g_j}e^{-\psi}   d\lambda
 \leq  
 \int_{D} \langle \text{Hess}_{\psi}^{-1} f,f \rangle_{g_j}  e^{-\psi} d\lambda 
 \leq \int_{D} \langle \text{Hess}_{\psi}^{-1}f,f \rangle_{g}e^{-\psi}   d\lambda,
$$
this means $\{u_j \}$ is uniformly bounded above in $L^2$, 
by Algaoglu's Theorem, 
there exists a subsequence $\{u_{c_j} \}$ of  $\{u_j\}$ converging weakly to $u$ in $L^2$ with $du=f$ and 
$$
\int_{D}  |u|^2_{g}  e^{-\psi} d\lambda \leq  \underset{ j \ra +\infty}{\lim \inf} \int_{D} | u_{j}|^2_{g^j}e^{-\psi}   d\lambda \leq \int_{D} \langle \text{Hess}_{\psi}^{-1} f,f \rangle_{g}  e^{-\psi} d\lambda,
$$
Theorem \ref{thm:cha nakano positive} shows that
	$(E,g)$ is Nakano $q$-semipositive.
\end{proof}

\section{Properties of curvature operator $\Re$}\label{sec:Properties of curvature operator Re}

Assume $M$ is an $n$-dimensional closed oriented manifold,
let $g^{TM}$ be the complete Riemannian metric on $TM$.
Let $\nabla^{TM}$ be the Levi-Civita connection 
associated to the metric $g^{TM}$,
it induce a canonical connection $\nabla^{\wedge(T^{*}M)}$ 
on $\wedge(T^{*}M)$.
Set
$$
R_{XY}=\nabla_X \nabla_Y-\nabla_Y\nabla_X-\nabla_{[X,Y]}
$$
be the Riemannian curvature tensor of 
the Levi-Civita connection.

Let $\{e_{1},\cdots,e_n\}$ be the 
oriented orthonormal basis of $TM$, 
and $\{\omega^{1},\cdots,\omega^n \}$ is 
the corresponding dual basis 
in $T^* M$ determined by $g^{TM}$.

For any $ 0\leq q\leq n$,
$\Omega^{q}(M)=\wedge^{q} (T^* M)$ 
denotes the space of smooth $q$-form on $M$.
$\Omega^{q}_{c}(M)$ are the elements in $\Omega^{q}(M)$
 with compact support.
We may express the $q$-form $\alpha$ and 
$\beta$ as:
$$
\begin{aligned}
	\alpha &=\sum_{i_1 <\cdots <i_q} \alpha _{i_1 \cdots i_q}  \omega^{i_1} \wedge \cdots \wedge \omega^{i_q}=\sum_{|I|=q} \alpha_{I} \omega^{I} \\
	\beta&=\sum_{i_1 <\cdots <i_q} \beta _{i_1 \cdots i_q}  \omega^{i_1} \wedge \cdots \wedge \omega^{i_q}=\sum_{|I|=q} \beta_{I} \omega^{I},
\end{aligned}
$$
smooth $0$-form is
 just smooth function defined on $M$.

It's well-known that 
the exterior differential operator $d$ and 
the codifferential operator $\delta$ satisfy:
$$
\begin{cases}
\displaystyle
d=\sum^{n}_{i=1} \omega^{i}\wedge \nabla^{\wedge(T^{*}M)}_{e_{i}} :
\Omega^{q}(M) \ra \Omega^{q+1}(M) 
\\
\displaystyle \delta=-\sum^{n}_{i=1} i_{e_{i}}\nabla^{\wedge(T^{*}M)}_{e_{i}}: 
\Omega^{q}(M) \ra \Omega^{q-1}(M) 
\end{cases},
$$
where $i_{e_{i}}$ denotes the interior product.

For any smooth function $\varphi \in C^{\infty}(M)$,
 we define the following point-wise inner product:
 \begin{equation}
 	\langle \alpha,\beta \rangle= 
	\sum_{i_1 <\cdots <i_q} \alpha _{i_1 \cdots i_q}\cdot \beta_{i_1 \cdots i_q},
	\end{equation}
the global weighted inner product:
\begin{equation}\label{eq:weighed-innder-product-global-manifold}
\langle \langle \alpha, \beta \rangle \rangle_{\varphi}
=\int_{M} \alpha \wedge * \beta e^{-\varphi}
=\int_{M} \langle \alpha,\beta \rangle  e^{-\varphi}dV_{g},
\end{equation}
and the corresponding norm:
$$
\begin{aligned}
 |\alpha|^2 &=\langle \alpha, \alpha\rangle, 
  \\
	\|\alpha \|^2_{\varphi}&=
\langle \langle \alpha, \alpha \rangle \rangle_{\varphi}.
\end{aligned}
$$

We set $L^2_{q}(M,\varphi)$ be 
the completion of the space $\Omega^{q}(M)$ with respect to the weighted inner product $\langle \langle \cdot, \cdot \rangle \rangle_{\varphi}$.
We use the same notation $d$ to denote 
the maximal (weak) differential operator 
$d: L^2_{q}(M,\varphi) \ra L^2_{q+1}(M,\varphi)$ for any $0 \leq q \leq n-1$,
let $d^{*}_{\varphi}$ 
denote the formal adjoint of $d$ with respect to 
the weighted inner product
 $\langle \langle \cdot, \cdot \rangle \rangle_{\varphi}$, 
simple calculations (\cite[Lemma 7.3]{Ji-Liu-Yu-2014}) shows that it satisfy:
$$
d^{*}_{\varphi}= e^{\varphi}\circ \delta \circ  e^{-\varphi}.
$$

\begin{defn}
The de-Rham Hodge Laplacian operator $\square$ associated to the Riemannian metric $g^{TM}$ is defined by:
$$
\square=d\delta+\delta d: \Omega^{q}(M) \ra \Omega^{q}(M),
$$
for any $0\leq q \leq n$.

The Bochner-Laplacian operator associated to the Riemannian metric $g^{TM}$ is defined by:
\begin{equation}
\triangle^{\wedge(T^{*}M)}_{0}=\sum^{n}_{i=1}\left( \nabla^{\wedge(T^{*}M)}_{\nabla^{TM}_{e_{i}}e_{i}}-\nabla^{\wedge(T^{*}M)}_{e_{i}}\nabla^{\wedge(T^{*}M)}_{e_{i}} \right):\Omega^{q}(M) \ra \Omega^{q}(M),
\end{equation}
for any $0\leq q \leq n$.
\end{defn}

\begin{rem}

\begin{itemize}
	\item[(1)] Both Laplace operators are 
	independent of the choice of the frame field.
	
\item[(2)] The comparison between
 the de-Rham Hodge Laplacian and the
Bochner-Laplacian 
induce the following zero-order differential operator
\begin{equation}
\Re_{q}:\square-\triangle^{\wedge(T^{*}M)}_{0}:
	\Omega^{q} \ra \Omega^q
\end{equation}
for any $1 \leq q \leq n$.
It's clear $\Re_q=0$ when $g$ is flat.
\end{itemize}
\end{rem}

\begin{defn}
	Let $\varphi$ be a function which
	is $C^2$ in a neighborhood of a point $x_0 \in M$,
	let $V,W$ be $C^{\infty}$ vector fields defined in a neighborhood of $x_0$,
	then $\text{Hess}(\varphi)$ is defined by:
	$$
	\text{Hess}(\varphi)(e_i,e_j)=
	\varphi_{ij}=
	e_i(e_j\varphi)-(\nabla^{TM} _{e_i}e_j) \varphi.
	$$
	and
	$$
	\text{Hess}_{\varphi}=\sum^{n}_{i,j=1} 
	\varphi_{ij} i_{e_{i}} \omega^{j} \wedge  :
	 \Omega^{q} \ra \Omega^{q}.
	$$
\end{defn}

For any relative compact domain $G$ in $M$ with smooth boundary defining function $\tau:M \ra \mr$,
we denote by $\Omega^{q}(G)$ the space of smooth $q$-form defined on $G$.
$\Omega_{c}^{q}(G)$ is the subspace of 
$\Omega^{q}(G)$ whose elements has compact support,
$\Omega^{q}(\bar{G})$ is the subspace of  
$\Omega^{q}(G)$ whose elements are smooth up to the boundary.
It was shown in \cite{Ji-Liu-Yu-2014} that 
\begin{equation}\label{eq:d-neumann-condiction}
	\alpha \in \wedge^{q}(\bar G) \cap \text{Dom}(d^{*}_{\varphi})  \Longleftrightarrow  (\nabla\tau  \lrcorner \alpha)_{x}=0 \;\; \text{for all x} \in \pa G.
\end{equation}
Note that the right hand side of \eqref{eq:d-neumann-condiction} is independent of the choice of the weight function.

\begin{lem}(\cite[Lemma 7.3]{Ji-Liu-Yu-2014})\label{lem:Bochner-indetity-d-operator-boundary-term-manifold}
	Let $G \subseteq M$ be a relative compact domain with boundary defining function $\tau \in C^{\infty}(\bar{G})$,
	for any smooth function $\varphi$ defined on $M$,
	we have the following identity 
\begin{equation}
\begin{aligned}
\|d\alpha \|^2_{\varphi}+\|d^{*}_{\varphi}\alpha \|^2_{\varphi}&=\int_{G} |\nabla \alpha  |^2 e^{-\varphi} dV_g 
	+\int_{G} \langle \text{Hess}_{\varphi} \alpha,\alpha  \rangle e^{-\varphi} dV_g \\
	&+\int_{M} 
	\langle 
		\sum^{n}_{i,j=1} \omega^{i} \wedge i_{e_{j}} R_{e_i e_j}
	\alpha, \alpha 
	\rangle e^{-\varphi}
	 dV_g \\
	&+ \int_{\pa G} \langle \text{Hess}_{\tau} \alpha,\alpha \rangle  e^{-\varphi} dS
\end{aligned}
	\end{equation}
	holds for any $\alpha \in \Omega^{q}(\bar{G}) \cap \text{Dom}(d^*_{\varphi})$.
\end{lem}

For the curvature term:
$\langle 
\sum^{n}_{i,j=1} \omega^{i} \wedge i_{e_{j}} R_{e_i e_j}
	\alpha, \alpha 
\rangle,$
the argument by I.M.Singer presented in 
\cite[Theorem 3.3]{Wu-Book-Bochner} shows that:
$$
\begin{aligned}
	\langle 
		\sum^{n}_{i,j=1} \omega^{i} \wedge i_{e_{j}} R_{e_i e_j}
	\alpha, \alpha 
	\rangle&= 
\langle 
	R(\rho\Re\rho^*) \alpha,
	\alpha
\rangle \\
&=\sum_{I}
\langle 
	(\rho\Re\rho^*)
	 (\alpha\wedge \omega^{I}),
	(\alpha\wedge \omega^{I})
\rangle  \\
&=\sum_I \langle
	\Re \rho^*
	(\alpha\wedge \omega^{I}),
	\rho(\alpha\wedge \omega^{I}) 
	\rangle \\
	&=\sum_I \langle 
\Re\rho^* (\alpha \wedge \omega^I),
\rho^* (\alpha\wedge \omega^I)
	\rangle 
\end{aligned}
$$
this means the positivity of the inner product 
$\langle 
		\sum^{n}_{i,j=1} \omega^{i} \wedge i_{e_{j}} R_{e_i e_j}
	\alpha, \alpha 
	\rangle$ 
is equivalent to the positivity of the curvature operator 
$\Re$.
See \cite[Theorem 3.3]{Wu-Book-Bochner} for more details of the definition of the operator $\rho,\rho^*,R$, 
here we won't restate them.

\section{Proof of Theorem \ref{thm:main-results-manifold} and Theorem  \ref{thm:deformation-Riemannian-metirc-manifold}}\label{sec:main-results-manifold}

\begin{thm}
Let $(M,g)$ be a complete Riemannian manifold,
for any point $p \in M$,
if there exists a relatively compact neighborhood $U$ of $x_0$,
for any smooth closed $1$-form $f \in \Omega^{1}_{c}(U)\cap \text{Ker}(d)$ with compact support, 
for any smooth function $\varphi$ defined on $\bar{U}$,
there exists a $0$-form $u \in L^2(U,\varphi)$ satisfy the 
such that 
$$
\int_{U} \| u\|^2 e^{-\varphi} dV_{g} 
\leq 
\int_{U} \langle \text{Hess}^{-1}_{\varphi}f,f\rangle e^{-\varphi} dV_{g},
$$
provided the right hand side is finite,
then the curvature operator $\Re$ associated to the Riemannian metric $g$ is semi-positive definite at $p$.
\end{thm}

\begin{proof}
Take $U$ be the normal coordinate neighborhood of $p$ 
with the coordinate systems $\{x_1, \cdots, x_n \}$ satisfy $p=(0,\cdots,0)$.

Set $$\varphi=\sum^{n}_{i=1} |x^2_{i}|-r^2_{0},$$
here $r_{0}>0$ is a positive number 
such that the sub-level set $U_{r_{0}}=\{x \in M; \varphi \leq 0 \}$ is relative compact in $U$,
it's clear $\varphi$ is a smooth strictly convex function defined on $U$.

By the assumption,
for any $\alpha \in \Omega^{1}_{c}(U)$, we have
$$
\begin{aligned}
 \langle  \langle \alpha,f \rangle   \rangle_{\varphi}^2 &= \langle  \langle \alpha ,du \rangle   \rangle_{\varphi}^2=\langle   \langle   d_{\varphi}^{*}\alpha ,u \rangle  \rangle_{\varphi}^2 \\
&\leq \langle   \langle  d_{\varphi}^{*}\alpha,d_{\varphi}^{*}\alpha\rangle  \rangle_{\varphi}  \cdot  \langle   \langle   u ,u \rangle  \rangle_{\varphi}.
\end{aligned}
$$
By the estimate of $\|u \|_{\varphi}^2$ and Lemma \ref{lem:Bochner-indetity-d-operator-boundary-term-manifold}, 
we have 
$$
\begin{aligned}
\langle   \langle  \alpha ,f \rangle  \rangle_{\varphi} ^2 \leq&  
\langle \langle  \text{Hess}^{-1}_{\varphi}f,f\rangle \rangle_\varphi \\
&\times \left(\int_{M} |\nabla \alpha|^2 e^{-\varphi}+\int_{M}\langle \text{Hess}_{\varphi} \alpha,\alpha \rangle e^{-\varphi} +\int_{M}\langle  \Re_{1}\alpha, \alpha \rangle e^{-\varphi} -\int_{M} |d \alpha|^2 e^{-\varphi} \right) \\
\leq&  
\langle \langle  \text{Hess}^{-1}_{\varphi}f,f\rangle \rangle_\varphi \\
&\times \left(\int_{M} |\nabla \alpha|^2 e^{-\varphi}+\int_{M}\langle \text{Hess}_{\varphi} \alpha,\alpha \rangle e^{-\varphi} +\int_{M}\langle  \Re_{1}\alpha, \alpha \rangle e^{-\varphi} \right)
\end{aligned}
$$
set
$$
f=\text{Hess}_{\varphi} \alpha
$$
then we have:
\begin{equation}\label{eq:main-inequality-manifold-curvature-operator}
	\int_{U} |\nabla \alpha|^2 e^{-\varphi}+\int_{U}\langle  \Re_{1}\alpha, \alpha \rangle e^{-\varphi} \geq 0
	\end{equation}
holds for any $\alpha \in \Omega^{1}_{c}(U)$.

We now argue by contradiction,
if the curvature operator $\Re$ is not semi-positive definite at $p$,
then there exist a $1$-form $\xi=\sum \xi_idx_i \in \wedge^{1}T^{*}_{p}(U)$,
a constant $c>0$ 
such that
$$
\langle \Re_{1} \xi_0, \xi_0  \rangle<-c <0,
$$
then the form 
$\xi=\sum \xi_i dx_i \in  \wedge^{1}T^{*}_{p}(U)$ is
a smooth with constant coefficients.

Take function 
$$v_0=\sum \xi_i x_i \in C^{\infty}(U)$$ 
then we have $dv_0=\xi$, 
set 
$$
v(x)=v_0\cdot \chi(x),
$$
viewed as a smooth function on $U$,
where $\chi \in C^{\infty }_{c}(U_{r_{0}})$ satisfying $\chi (x)=1$ for $x\in U_{\frac{r_0}{2}}$. Set $\alpha =dv$, then $\alpha \in \Omega^{1}_{c}(U)$ and 
$$
\langle \Re_1 \alpha,\alpha \rangle(x)<-c<0 
$$
holds for any $x\in U_{\frac{r_0}{2}}$.

Set 
$$
\varphi_{s} =s(\sum^{n}_{i=1} |x^2_{i}|-\frac{ r^2_{0}}{4}),
$$
Replacing $\psi$ in the left hand of \eqref{eq:main-inequality-manifold-curvature-operator} by $\varphi_s$,
we have:
$$
	\int_{U} |\nabla \alpha|^2 e^{-\varphi_s}+\int_{U}\langle  \Re_{1}\alpha, \alpha \rangle e^{-\varphi_s} \geq 0
$$
for any $s>0$.

By the construction 
of $\alpha$, 
there exists a constant $C>0$
such that $|\nabla \alpha|(x) \leq C$ 
holds for any $x \in U$ 
and $|\nabla \alpha|(x)=0$ for any $x\in U_{\frac{r_0}{2}}$.

Notice that 
$\displaystyle \lim_{s \ra +\infty} \varphi_s =+\infty$ when 
$x \in U \backslash U_{\frac{r_0}{2}}$,
and $\varphi_s \leq 0 $ when $x\in U_{\frac{r_0}{2}}$,
then we have:
$$
\begin{aligned}
 \int_{U} |\nabla \alpha|^2 e^{-\varphi_s}+\int_{U}\langle  \Re_{1}\alpha, \alpha \rangle e^{-\varphi_s} 
  &\leq 
-c\int_{U_{\frac{r_0}{2}}} e^{-\varphi_s} + \int_{U \backslash U_{\frac{r_0}{2}}}\langle  \Re_{1}\alpha, \alpha \rangle e^{-\varphi_s}
\\
&+C \int_{U \backslash U_{\frac{r_0}{2}}} e^{-\varphi_s} +0<0
\end{aligned}
$$
when $s \ra +\infty$,
this contradicts to  \eqref{eq:main-inequality-manifold-curvature-operator}.
Thus $\Re$ is semi-positive definite at $p$.
\end{proof}

Recall the definite of $q$-convex manifold.

\begin{defn}
Assume $(M,g)$ is a complete Riemannian manifold,
we say $M$ is a $q$-convex manifold if it admits a smooth exhaustive strictly $q$-convex function.
\end{defn}

Now we are ready to prove 
Theorem  \ref{thm:deformation-Riemannian-metirc-manifold}.

\begin{thm}
Assume $(M,g_t)$ is a complete homogeneous $1$-convex Riemannian manifold,
where $g_{t}$ is an increasing sequence of complete Riemannian metric with semi-positive definite the curvature operator $\mathscr{R}^{t}$ on $M$  
and converge to a Riemannian metric $g$,
then the curvature operator $\mathscr{R}$ associated to $g$ is semipositive definite on M.
\end{thm}

\begin{proof}
	By the assumption,
for any smooth function $\psi$ defined on $M$
	we have 
	$$
	\langle \langle \text{Hess}_{\psi}^{-1}f,f \rangle\rangle_{g_{j}} \leq	\langle \langle \text{Hess}_{\psi}^{-1}f,f \rangle\rangle_{g}
	$$
holds for all $f\in \Omega^{1}_{c}(M)\cap \text{Ker}(d)$ and any $j \in \mathbb{N}^*$, provided the right hand side is finite.

For each $ j$,
Theorem 7.1 in \cite{Ji-Liu-Yu-2014} shows that,
there exists $u_j \in L^2_{0}(M,\psi)$ such that $du_j=f$ with 
$$
 \int_{M} | u_{j}|^2_{g_j}e^{-\psi}  
 \leq  
 \int_{M} \langle \text{Hess}_{\psi}^{-1} f,f \rangle_{g_j}  e^{-\psi} 
 \leq \int_{D} \langle \text{Hess}_{\psi}^{-1}f,f \rangle_{g}e^{-\psi} ,
$$
this means $\{u_j \}$ is uniformly bounded above in $L^2$, 
by Algaoglu's Theorem, 
there exists a subsequence $\{u_{c_j} \}$ of  $\{u_j\}$ converging weakly to $u$ in $L^2$ with $du=f$ and 
$$
\int_{M}  |u|^2_{g}  e^{-\psi} \leq  \underset{ j \ra +\infty}{\lim \inf} \int_{M} | u_{j}|^2_{g^j}e^{-\psi}  \leq \int_{M} \langle \text{Hess}_{\psi}^{-1} f,f \rangle_{g}  e^{-\psi},
$$
Corollary \ref{cor:postive-curvature-operator-manifold}
 shows that the curvature operator $\Re$ semipositive definite on $M$.
\end{proof}

\section{The boundary convexity on domain in $1$-convex manifold}
\label{sec:boundary-Riemannian-manifold}

We prove Theorem \ref{thm:char-q-convex-manifold}
in this section. We need the following lemma.

\begin{lem}(
\cite[Theorem 9.51]{GTM218} \label{lem:exist-d-1-form-boundary})
Let  $M$ be a  smooth  manifold with real dimensional $n$,
	for a hypersurface $S$ in $M$ with smooth defining function: $\rho: M \ra \mr$, where $|\nabla \rho (x)| \neq 0$ for all $x \in S$.
	For any smooth function $\varphi:S \ra \mr$, 
	there exists a unique solution $u$ to the following equation with Cauchy values.
	\begin{equation}
		\begin{cases}
			\nabla \rho \cdot \nabla u =0 \\
			u|_{S}=\varphi
					\end{cases}
	\end{equation}
\end{lem}

\begin{thm}[=Theorem \ref{thm:char-q-convex-manifold}]
	Suppose $(M,g)$ is a $n$-dimensional complete homogeneous $1$-convex Riemannian manifold with semi-positive definite curvature operator $\Re_q$.
	Let $G$ be a relatively compact domain in $M$ with smooth boundary defining function $\tau: M \ra \mr$.
	If for any closed $1$-form smooth up to the boundary $f \in \Omega^{q}(\bar{G})  \cap \text{Dom}(d^*) \cap \text{Ker}(d)$, 
for any strictly convex function $\psi$ on $M$,
there exists a solution $u\in L^2_{q-1}(G, \psi )$ satisfying  $du=f$
and
	$$
	\int_{G} |u|^2 e^{-\psi}d\lambda  \leq \int _{G} \langle  \text{Hess}^{-1}_{\psi} f, f \rangle e^{-\psi}d\lambda  ,
	$$	
provided the right hand side is finite,
then $\pa G$ is $q$-convex.
\end{thm}

\begin{proof}
We begin our argument by contradiction,
if $\pa G$ is not $q$-convex,
for any 
	non-$q$-convex boundary point 
	$p$ in $\pa G$,
	we construct a strictly convex function that is nondegenerate at $p$ in the following steps.

Since $M$ is homogeneous $1$-convex,
there exists a  smooth strictly 
convex function 
$\varphi$ such that $p$ is 
the nondegenerate 
critical point of 
$\varphi$,
we can assume
$\varphi(p_0)=0$ 
by adding some constants,
for any $r_0>0$,
set $U_{r_0}
=\{x \in M;\varphi(x)<r_0 \}$.

For the normal coordinate 
neighborhood $\mathcal{O}_p(\delta_p)$ of $p$
with
the coordinate systems 
$\{x_1, \cdots, x_n \}$,
for any smooth geodesic 
$\gamma:(-t,t) \ra M$ connecting $p$ and $x$ in $\mathcal{O}_p(\delta_p)$,
since $\frac{\pa }{\pa x_j} \gamma$ is parallel along $\gamma$,
we have:
$$
0=\nabla_{\gamma '} \frac{\pa }{\pa x_j}
=\sum x_i\nabla_{\frac{\pa \gamma}{\pa x_j} } \frac{\pa }{\pa x_j}
=\sum x_i\overline{\Gamma}^{k}_{ij} \gamma (t) \frac{\pa \gamma(t)}{\pa x_k}
$$
thus 
$$
\sum^{n}_{k=1}x_k \overline{\Gamma}^{k}_{ij}(x)=0
$$
holds for any $x \in \mathcal{O}_p(\delta_p)$,
then for the function 
$$
\psi=\sum^{n}_{i=1} x^2_i,
$$
we have
$$
	\begin{aligned}
		\text{Hess}(\psi)(\frac{\pa}{\pa x_i},\frac{\pa}{\pa x_j}) 
   &=2\delta_{ij}-2\sum^{n}_{k=1}x_k \overline{\Gamma}^{k}_{ij} \\
   &=2\delta_{ij},
	\end{aligned}
$$
thus for any $q$-form 
$\alpha \in \Omega^q (\mathcal{O}_p(\delta_p))$ 
we have 
\begin{equation}\label{eq:hessian-calculation}
	\langle \text{Hess}_{\psi} \alpha ,\alpha  \rangle (x)
 = 2 q \langle \alpha, \alpha \rangle (x)
\end{equation}
holds for any 
$x \in\mathcal{O}_p(\delta_p)$.
We take $r_0$ small enough 
such that 
$U_{r_0} \subset \subset 
\mathcal{O}_p(\delta_p)$,
set
\begin{equation}\label{eq:construct-weight}
\tilde{\varphi}(x)=
\begin{cases} 
\psi-r_0 & x \in U_{r_0} \\
\max \{\psi-r_0 ,\varphi-r_0 \} 
& x \in \mathcal{O}_p(\delta_p) \backslash U_{r_0} \\
\varphi-r_0 
& x \in M \backslash \mathcal{O}_p(\delta_p)
\end{cases}
\end{equation}
then $\tilde{\varphi}(x)$ is 
strictly convex on $M$,
by the smooth theorem of strictly convex function (\cite{Green-Wu-1978}) on the manifold,
we may assume $\tilde{\varphi}$ is smooth, 
set
$\widetilde{U_{r_0}}
=\{x \in M;\widetilde{\varphi(x)}<0 \}$.

We choose a smooth tangent 
	orthonormal tangent frame field 
	$\{e_1,...,e_n \}$ in a neighborhood of $p$ 
	such that $e_n=N$ 
	(the out normal vector)
	on the boundary such that 
	at the $p$ we have:
\begin{enumerate}[(i)]
\item $p_0$ is the origin
	\item $\nabla_{e_j}N=k_j e_j$ for $j=1,\cdots,(n-1)$ for some positive numbers.
	For convenience,
we assume 
$$k_1 \leq \cdots \leq k_{n-1}.$$
\end{enumerate}
the assumption shows that
there exists 
a tangent $q$-form
 $\xi=
 \omega^1\wedge \cdots \wedge \omega^q 
 \in \wedge^{q} T^*_{p} \pa G$ 
such that 
$$
\langle \text{Hess}_{\tau} \xi,\xi \rangle
=(\sum^{q}_{i=1}k_i) |\xi|^2<0.
$$

Lemma \ref{lem:exist-d-1-form-boundary} shows that,
for any $1 \leq i \leq q$, 
there exists $f_i \in \Omega^{1}(\bar{G}) \cap \text{Dom}(d^*) \cap \text{Ker}(d)$ such that $f_i(O)=\omega^i$,
set
$$
f=f_1 \wedge \cdots \wedge f_q,
$$
we have $f \in \Omega^{1}(\bar{G}) \cap \text{Dom}(d^*) \cap \text{Ker}(d)$ and $f (O)= \omega^1\wedge \cdots \wedge \omega^q$.
By the continuity,
there exists a constants $c<0$  such that:	
$$
\langle \text{Hess}_{\tau}f, f  \rangle (x) <-c
$$
holds for all $x\in  \widetilde{U_{r_0}} \cap \pa G$.

Then
we have 
$$
\begin{aligned}
 \langle  \langle f ,f \rangle   \rangle_{\tilde{\varphi}}^2 
 &= \langle  \langle f,du \rangle   \rangle_{\tilde{\varphi}}^2
 =\langle   \langle   d_{\tilde{\varphi}}^{*}f ,u \rangle  \rangle_{\tilde{\varphi}}^2 \\
&\leq \langle   \langle  d_{\tilde{\varphi}}^{*}f ,d_{\tilde{\varphi}}^{*}f \rangle  \rangle_{\tilde{\varphi}}  \cdot  \langle   \langle   u ,u \rangle  \rangle_{\tilde{\varphi}}
\end{aligned}
$$
combine this with the estimate of $\| u\|^2_{\tilde{\varphi}}$ and Lemma \ref{lem:Bochner-indetity-d-operator-boundary-term-manifold},
we have 
$$
\begin{aligned}
\langle   \langle  f,f \rangle  \rangle_{\tilde{\varphi}} ^2 \leq&  
\langle   \langle   \text{Hess}_{\tilde{\varphi}}^{-1} f ,f \rangle  \rangle_{\tilde{\varphi}} \cdot\\
&(
\int_{G} |\nabla f |^2 e^{-\tilde{\varphi}}dV_g
 +\int _{G} \langle    \text{Hess}_{\tilde{\tilde{\varphi}}} f,f\rangle e^{-\tilde{\varphi}}dV_g
 +\int _{G} \langle    \Re_q f,f\rangle e^{-\tilde{\varphi}}dV_g \\
 &+\int_{\pa  G}\langle  \text{Hess}_{\tau} f, f \rangle   e^{-\tilde{\varphi}}dS) \\
\end{aligned}
$$

From the assumption in the theorem,
we may assume that
there exist constants
$\lambda_{1}, \lambda_n>0$
such that 
$$
\lambda_1 |f|^2
\leq 
\langle \text{Hess}_{\varphi }f,f \rangle (x)
\leq 
\lambda_n|f|^2
$$
holds for any $x \in G \backslash \widetilde{U_{r_0}}$, 
we may assume $\lambda_1 >2q$ by multiply positive
constant with $\varphi$ in \eqref{eq:construct-weight}.
We can also find a constant $t_n>0$ such that 
$$
\langle \Re_q f,f \rangle \leq t_n 
$$
holds for any $x \in G$.
Then we have: 
$$
\begin{aligned}
\langle   \langle  f,f \rangle  \rangle_{\tilde{\varphi}} ^2 \leq&  
\langle   \langle   \text{Hess}_{\tilde{\varphi}}^{-1} f ,f \rangle  \rangle_{\tilde{\tilde{\varphi}}} \cdot\\
&(
\int_{G} |\nabla f |^2 e^{-\tilde{\varphi}}dV_g
 +\int _{G} \langle  \text{Hess}_{\tilde{\varphi}} f,f\rangle e^{-\tilde{\varphi}}dV_g 
\\
 &+\int _{G} \langle    \Re_q f,f\rangle e^{-\tilde{\varphi}}dV_g 
 +\int_{\pa  G}\langle  \text{Hess}_{\tau} f, f \rangle   e^{-\tilde{\varphi}}dS )\\
 &\leq (\frac{1}{\lambda_1}\int_{G\backslash \widetilde{U_{r_0}}}|f|^2 dV_g 
 +\frac{1}{2q}\int_{\widetilde{U_{r_0}}} |f|^2 dV_g) \cdot  \\
 & (
\int_{G} |\nabla f|^2 e^{-\tilde{\varphi}}dV_g\\
 &+\lambda_n\int _{G\backslash \widetilde{U_{r_0}}}|f|^2 e^{-\tilde{\varphi}}dV_g 
 +2q\int_{\widetilde{U_{r_0}}} |f|^2 dV_g\\
 &+t_n \int _{G} |f|^2 e^{-\tilde{\varphi}}dV_g 
 +\int_{\pa  G}\langle  \text{Hess}_{\tau} f, f \rangle e^{-\tilde{\varphi}}dS
 )
\end{aligned}
$$
$$
\begin{aligned}
	&\leq \frac{1}{2q} \langle \langle f,f \rangle \rangle_{\tilde{\varphi}} \cdot \\
 & (
	\int_{G} |\nabla f|^2 e^{-\tilde{\varphi}}dV_g\\
	 &+\lambda_n\int _{G\backslash \widetilde{U_{r_0}}}|f|^2 e^{-\tilde{\varphi}}dV_g 
	 +2q\int_{\widetilde{U_{r_0}}} |f|^2 dV_g\\
	 &+t_n \int _{G} |f|^2 e^{-\tilde{\varphi}}dV_g 
	 +\int_{\pa  G}\langle  \text{Hess}_{\tau} f, f \rangle e^{-\tilde{\varphi}}dS
	 ),
\end{aligned}
$$

this means:
\begin{equation}\label{eq:boundary-main-1}
\int_{G} |\nabla f |^2 e^{-\tilde{\varphi}}dV_g
 +(\lambda_n -2q)\int _{G\backslash\widetilde{U_{r_0}} }  |f|^2 e^{-\tilde{\varphi}}dV_g 
+t_n\int _{G}|f|^2 e^{-\tilde{\varphi}}dV_g 
 +\int_{\pa  G}\langle  \text{Hess}_{\tau} f, f \rangle   e^{-\tilde{\varphi}}dS
 \geq 0
\end{equation}

Substitute $\tilde{\varphi}$ in \eqref{eq:boundary-main-1}
with $\tilde{\varphi}_s=s\tilde{\varphi}$, 
since $G$ is relatively compact, 
then there exists a constant $C>0$
such that 
$$ 
|f|^2 \leq   C, |\nabla f |^2 \leq C \;\; \text{for } \forall x \in G
$$
notice that
$\tilde{\varphi}_s(x) <0$ for all $x\in \widetilde{U_{r_0}}$ and $\tilde{\varphi}_s(x)\ra +\infty$ for all $x\in G \backslash \tilde{U_{r_0}}$,
then there exists constant $M_0, M_1>0$ such that:
\begin{equation}\label{eq:main-inequality-1-Riemannian-manifold}
\frac{sM_0}{e^s}+M_1 \int_{G\cap \widetilde{U_{r_0}}}e^{-\tilde{\varphi}_s}dV_g -c  \int_{\pa G\cap \widetilde{U_{r_0}}}  e^{-\tilde{\varphi}_s} dS  \geq 0.
\end{equation}
holds for any $s>1 ,r_0 >0$.

Next,
we estimate the integration in 
\eqref{eq:main-inequality-1-Riemannian-manifold} 
by terms.
For the first term,
we have:
$$
M_1  \int_{G\cap \widetilde{U_{r_0}}}e^{-\tilde{\varphi}_s}dV_g
 \leq M_1 \int_{\widetilde{U_{r_0}}}e^{-\tilde{\varphi}_s}dV_g.
$$

For the second integration in 
\eqref{eq:main-inequality-1-Riemannian-manifold},
one can verify that $O$ is 
the nondegenerate critical point of 
the boundary defining function $\tau$ by 
contracting $r_0$ if necessary.
Since $p_0$ is 
the nondegenerate critical point 
of the strictly convex function 
$\tilde{\varphi}$,
let 
$$
\sigma:\pa G\cap \widetilde{U_{r_0}} \ra \pa \widetilde{U_{r_0}}
$$
be the diffeomorphism,
then there exists constants $c_1,c_2>0$ such that:
$$
c  \int_{\pa G\cap \widetilde{U_{r_0}}}  e^{-\tilde{\varphi}_s} dS \geq c\cdot c_1  \int_{\pa  \widetilde{U_{r_0}}}  e^{-\tilde{\varphi}_s} dS =  c_2 r^{-1}_{0}  \int_{\widetilde{U_{r_0}}}  e^{-\tilde{\varphi}_s},
$$
the last equal sign follows from
the Fubini's theorem on the Riemannian manifold.

Thus when $r_0 \leq \frac{c_2}{M_1 +1}$,
we have:
$$
\begin{aligned}
\frac{sM_0}{e^s}+M_1 \int_{G\cap \widetilde{U_{r_0}}} &e^{-\tilde{\varphi}_s}dV_g -c  \int_{\pa G\cap \widetilde{U_{r_0}}}  e^{-\tilde{\varphi}_s} dS \\
 &\leq M_0+(M_{1}-c_2 r^{-1}_0)\int_{\widetilde{U_{r_0}}}e^{-\tilde{\varphi}_s}dV_g \\
 &\leq M_0- \int_{\widetilde{U_{r_0}}}e^{-\tilde{\varphi}_s}dV_g
\end{aligned}
$$
the right hand side goes to $-\infty$ when $s \ra +\infty$,
this produce a contradiction to \eqref{eq:main-inequality-1-Riemannian-manifold},
thus $\pa G$ is $q$-convex.
\end{proof}

\bibliographystyle{amsplain}

\end{document}